\documentclass[a4paper,12pt]{article}

\usepackage[a4paper]{geometry}
\usepackage{amsmath,amssymb,amscd,amsfonts} \usepackage{bm}
\usepackage{bbm} \usepackage{xspace} \usepackage[T1]{fontenc}
\usepackage[latin1]{inputenc} \usepackage{indentfirst}
\usepackage{html} \usepackage{url} \usepackage{harvard}
\usepackage{setspace} \usepackage{todonotes} \usepackage{appendix}
\usepackage{epsfig,latexsym,xspace,subfigure}
\newtheorem{theorem}{Theorem}[section]
\newtheorem{corollary}[theorem]{Corollary}
\newtheorem{lemma}[theorem]{Lemma}

 \newtheorem{definition}{Definition}[section]
 
\newtheorem{example}{Example}[section]
\newtheorem{cond}{Condition}[section]
\newcommand{\halm}{\hspace*{\fill} $\Box$\par}
\newenvironment{proof}{\noindent {\bf
    Proof. }}{\halm\vspace{\baselineskip}} \newcommand{\nei}[1]{{\rm
    ne}(#1)}

\renewcommand{\bm}[1]{\boldsymbol{#1}} \newcommand{\bX}{{\bm X}}
\newcommand{\cX}{{\cal X}} \newcommand{\bx}{{\bm x}}
\newcommand{\bq}{{\bm q}} \newcommand{\bp}{{\bm p}}
 \newcommand{\R}{\mathbbm{R}}
\newcommand{\E}{\mbox{\rm E}}  \newcommand{\IF}{\mbox{\rm IF}}

\newcommand{\ie}{{\em i.e.\/}\xspace} \newcommand{\eg}{{\em
    e.g.\/}\xspace}

\newcommand{\thmref}[1]{Theorem~\ref{thm:#1}}
\newcommand{\secref}[1]{\S~\ref{sec:#1}}
\newcommand{\condref}[1]{\mbox{Condition~\ref{cond:#1}}}
\newcommand{\lemref}[1]{\mbox{Lemma~\ref{lem:#1}}}
\newcommand{\itref}[1]{\mbox{\ref{it:#1}}}
\newcommand{\appref}[1]{\mbox{Appendix~\ref{sec:#1}}}
\newcommand{\exref}[1]{\mbox{Example~\ref{ex:#1}}}
\def\T{{\footnotesize {^{_{\sf T}}}}} \newcommand{\Real}{{\rm
    I}\negthinspace {\rm R}}
\def\Bka{{\it Biometrika}}
\author{\small Philip Dawid$^{1}$, Monica Musio$^{2}$,
  and Laura Ventura$^3$\\[-1ex]
  \small$^1\,$University of Cambridge, UK\\[-.7ex]
  \small$^2\,$Universit\`a degli Studi di Cagliari, Italy\\[-1ex]
  \small$^3\,$Universit\`a degli Studi di Padova, Italy \\[-.7ex] }
\title{
  Minimum Scoring Rule Inference}

\begin{document}
\renewcommand{\thefootnote}{\arabic{footnote}}
\maketitle

\begin{abstract}
  Proper scoring rules are methods for encouraging honest assessment
  of probability distributions.  Just like likelihood, a proper
  scoring rule can be applied to supply an unbiased estimating
  equation for any statistical model, and the theory of such equations
  can be applied to understand the properties of the associated
  estimator.  In this paper we develop some basic scoring rule
  estimation theory, and explore robustness and interval estimation
  preoperties by means of theory and simulations.
\end{abstract}

\noindent {\em Keywords:} $B$-robustness; Bregman estimate; Composite
score; Godambe information; $M$-estimator; Pseudolikelihood; Tsallis
score; Unbiased estimating equation.

\section{Introduction}
\label{sec:intro}
Suppose we wish to fit a parametric statistical model $\{P_\theta:
\theta \in \Theta \subseteq \Real^p\}$, based on a random sample
$(x_1, \cdots,x_n)$ of size $n$.  The most popular tool for inference
on the parameter $\theta$ is the log-likelihood function, given by
\begin{equation}
  \label{eq:log-likelihood}
  \ell (\theta) = \sum_{i=1}^n \log{p_\theta(x_i)}
  \ ,
\end{equation}
where $p_\theta(x)$ is the density associated to $P_\theta$.  For
instance, the maximum likelihood estimator is defined as
$\widehat{\theta}= \arg\max_{\theta} \, \ell(\theta)$, and confidence
regions with nominal coverage $1-\alpha$ can be constructed as $\{
\theta: W(\theta) \leq \chi^2_{p;1-\alpha} \}$, where
$W(\theta)=2\{\ell(\widehat\theta)-\ell(\theta)\}$ is the likelihood
ratio statistic and $\chi^2_{p;1-\alpha}$ is the $(1-\alpha)$-quantile
of the $\chi^2_p$ distribution.

However, likelihood-based inference generally requires strict
adherence to the model assumptions, and can behave quite poorly under
slight model misspecification.  A possible solution is to resort to
suitable pseudo-likelihood functions, which are intended as surrogates
of the full likelihood.  Useful examples are given by composite
likelihoods (Cox and Reid, 2004, Varin {\em et al.\/}, 2011), when the
fully specified likelihood is computationally cumbersome or when a
fully specified model is out of reach, and by quasi-likelihoods, which
are derived from suitable unbiased estimating equations (see, among
others, McCullagh, 1991, Adimari and Ventura, 2002).

Both full and pseudo likelihood inference are special cases of a more
general estimation technique based on {\em proper scoring rules\/}
(see, \eg, Dawid and Musio, 2014), which are methods for encouraging
honest assessment of probability distributions.  In such a case, the
log-likelihood function is replaced by the function
\begin{equation}
  \label{eq:score}
  S(\theta)=\sum_{i=1}^n S(x_i,P_\theta)
  \ ,
\end{equation}
where $S( x, P)$ is a proper scoring rule, as described in
\secref{psr} below; this can be chosen to increase robustness, or for
ease of computation.  Minimising (\ref{eq:score}) will yield an
unbiased estimating equation, for any statistical model.

The appeal of scoring rules estimation lies in the potential adaption
of the scoring rule to the problem at hand, and it forms a special
case of $M$-estimation (see, \eg, Huber and Ronchetti, 2009).  In view
of this, under regularity conditions, asymptotic arguments indicate
that the estimator $\widehat\theta_S =\arg\min_\theta S(\theta)$ is
consistent and asymptotically normal, with asymptotic covariance
matrix given by the inverse of the Godambe information.  This allows
the construction of Wald type test statistics and confidence regions.
However, as is well known, Wald type statistics force confidence
regions to have an elliptical shape and may be less accurate for small
sample sizes.  On the other hand, the asymptotic distribution of the
likelihood ratio type statistics derived from (\ref{eq:score}) depart
from the familiar likelihood result, involving a linear combination of
independent chi-squared variates with coefficients given by the
eigenvalues of a matrix related to Godambe information.  As a
consequence, most routine statistical analyses employ Wald type
statistics.

The aim of this paper is to discuss inference based on proper scoring
rules.  Stemming from the failure of the information identity,
inference based on proper scoring rules requires suitable corrections.
In particular, when considering the scoring rule ratio statistic for a
parameter of interest, we discuss suitable adjustments that allow
reference to the usual asymptotic chi-square distribution.  Particular
focus is on robust proper scoring rules, {\em i.e.\/} scoring rules
that lead to estimators with bounded influence function.  Indeed, in
this case, the adjusted scoring rule ratio statistic can be used in
the usual way to derive confidence regions for a multidimensional
parameter of interest, while in general a quasi-likelihood does not
exist (McCullagh, 1991).

The paper is organized as follows.  In Section~2, background theory
and examples on proper scoring rules are given, while Section~3
focuses on proper scoring rule inference.  Section~4 discusses
asymptotic results on scoring rule procedures, and introduces the
adjustments of the scoring rule ratio statistic that allow reference
to the usual asymptotic chi-square distribution.  In Section~5
robustness properties of the scoring rules estimators are studied.  In
particular, conditions for robustness of the Bregman score are
investigated in detail.  Three examples dealing with confidence
regions from the adjusted scoring rule ratio statistic are analysed in
Section~6.  Simulation results indicate that such adjustments allow
accurate inferences, and it is argued that scoring rules have an
important role to play in frequentist inference.  Some concluding
remarks are given in Section~7.

\section{Proper scoring rules}
\label{sec:psr}

Let $X$ be a random variable taking values in a sample space ${\cX}$.
A {\em scoring rule\/} (see, \eg, Dawid, 1986) is a loss function
$S(x,Q)$ measuring the quality of a quoted probability distribution
$Q$ for $X$, in the light of the realised outcome $x$ of $X$.  It is
{\em proper\/} if, for any distribution $P$ for $X$, the expected
score $S(P,Q) :=\E_{X \sim P}\,S(X,Q)$ is minimised by quoting $Q =
P$.  Equivalently, the associated {\em divergence\/} or {\em
  discrepancy function\/} (Dawid, 1998), given by $D(P,Q) :=
S(P,Q)-S(P,P)$, is always non-negative.  There is a very wide variety
of proper scoring rules: for general characterisations see, among
others, McCarthy (1956), Savage (1971), and for various special cases
see Dawid (1998, 2007) and Gneiting and Raftery (2007).  We now
consider some of these in more detail.

Let $q(\cdot)$ denote the density of $Q$ with respect to an underlying
$\sigma$-finite measure $\mu$, or the probability mass function in the
discrete case.  Although greater generality is possible, in this paper
we will assume $\mu$ is Lebesgue measure for ${\cal X}$ a real
interval, and counting measure for ${\cal X}$ discrete.  For a finite
(especially binary) sample space ${\cal X}$, a useful proper scoring
rule is the {\em Brier\/} (Brier, 1950) or {\em quadratic\/} score
$S(x,Q) = \{1-q(x)\}^2 + \sum_{y \neq x} q(y)^2$, which is just the
squared Euclidean distance between the vector $\bq:= (q(y): y\in {\cal
  X})$ corresponding to $Q$, and the vector ${\bm\delta}_x$
corresponding similarly to the one-point distribution at $x$.  The
associated discrepancy $D(P,Q)$ is the squared Euclidean distance
between $\bp$ (the vector corresponding to $P$) and $\bq$.  Another
prominent proper scoring rule (Good, 1952) is the {\em log score\/}
$S(x, Q) = -\log q(x)$, whose associated discrepancy is the
Kullback-Leibler divergence $K(P,Q)$.  These are both special cases
(with, respectively, $\psi(t) \equiv t^2$ and $\psi(t) \equiv t \log
t$) of a general {\em separable Bregman score\/} construction (see
\eg\ Dawid, 2007, eq.~(16)):
\begin{equation}
  \label{eq:sepbreg}
  S(x,Q) = -\psi'\{q(x)\}  - \int \left[\psi\{ q(y)\} -  q(y) \psi'\{ q(y)\}\right]\,d\mu(y)
  \ ,
\end{equation}
where the {\em defining function\/} $\psi:\R^+\rightarrow \R$ is
convex and differentiable.  The associated {\em Bregman divergence\/}
is
\begin{equation}
  \label{eq:bregdiv}
  D(P,Q) = \int \Delta\left\{p(y),q(y)\right\}\,d\mu(y)
  \ ,
\end{equation}
where $\Delta(a,b) = \psi(a) - \psi(b) - \psi'(b)(a-b) \geq 0$ by
convexity.  Another important special case of this construction, the
{\em Tsallis score\/}, arises on taking $\psi(t) \equiv t^\gamma$
($\gamma>1)$.  This yields
\begin{equation}
  \label{eq:tsallisscore}
  S(x,Q) = (\gamma - 1) \int q(y)^\gamma\,d\mu(y) - \gamma q(x)^{\gamma-1}
  \ ,
\end{equation}
with divergence function
\begin{equation}
  \label{eq:tsallisdivergence}
  D(P,Q) = \int p(y)^\gamma\,d\mu(y) + (\gamma-1)\int q(y)^\gamma\,d\mu(y) - \gamma\int p(y) q(y)^{\gamma-1}\,d\mu(y).
\end{equation}
The {\em density power divergence\/} $d_\alpha$ of Basu {\em et al.\/}
(1998) is just \eqref{eq:tsallisdivergence}, with $\gamma = \alpha+1$
and $\mu$ given by Lebesgue measure, multiplied by $1/\alpha$.

In order to evaluate the log score, we only need to know the value of
the forecast density function, $q(\cdot)$, at the outcome $x$ of $X$
that Nature in fact produces.  So long as the size of ${\cal X}$
exceeds two, the log score is essentially the only proper scoring rule
that is {\em strictly local\/} in the above sense (Bernardo, 1979).
However, we can weaken the locality requirement, and so admit further
``local proper scoring rules''.  For a sample space ${\cal X}$ that is
an open subset of a Euclidean space, we ask that $S(x, Q)$ should
depend on the density function $q(\cdot)$ only through its value and
the value of a finite number of its derivatives at $x$.  For the case
that ${\cal X}$ is a real interval, Parry {\em et al.\/}~(2012) show
that any such local proper scoring rule is a linear combination of the
log score and what they term a {\em key local\/} scoring rule, which
they have characterised.  A key local scoring rule has the convenient
property that it can be computed without knowledge of the
normalisation constant of the density.  The simplest key local scoring
rule is that based on the proposal by Hyv\"arinen (2005),
\begin{equation}
  \label{eq:hyvscore}
  S_H(x,Q)=2\Delta\ln q(x)+ \left|\nabla\ln q(x)\right|^{2}
  \ ,
\end{equation}
where, in the case of a real sample space, $\nabla:=
(\partial/\partial x)$ and $\Delta :=
\partial^2/(\partial x)^2$.  Formula (\ref{eq:hyvscore}) can also be
applied to the case of a multivariate observation $X = (X_1, \ldots,
X_k)$, with $\nabla:= (\partial/\partial x_j)$ and $\Delta :=
\sum_{j=1}^k
\partial^2/(\partial x_j)^2$.  Further extensions to a general
Riemannian sample space are possible (see Dawid and Lauritzen, 2005).

\subsection{Composite scores}

In this section we consider the case of a multidimensional variable
$\bX$.  Let $\bX^*$ be a subvector of (or, more generally, a function
of) $\bX$, and let $S^*$ be a proper scoring rule for $\bX^*$.  Then
we can define a proper scoring rule $S$ for $\bX$ as $S(\bx, Q) :=
S^*(\bx^*,Q^*)$, where $Q^*$ denotes the marginal distribution of
$\bX^*$ when $\bX \sim Q$.  Alternatively, let $\bX^\dag$ denote
another subvector or function of $\bX$.  Then a proper scoring rule
can be generated as $S(\bx,Q) = S^*(\bx^*, Q^\dag)$, where $Q^\dag$
denotes the conditional distribution, when $\bX\sim Q$, of $\bX^*$,
given $\bX^\dag =\bx^\dag$.  By an abuse of language, we may refer to
the specification of $(\bX^*, \bX^\dag)$ as a {\em conditional
  variable\/}, $\bX_0$ say, and that of $Q^\dag$, for every value
$\bx^\dag$ of $\bX^\dag$, as its distribution, $Q_0$ say, and then we
write $\bX_0 \sim Q_0$.

Now let $\{\bX_k\}$ be a collection of marginal and/or conditional
variables, and let $S_k$ be a proper scoring rule for $\bX_k$.  Then
we can construct a proper scoring rule for $\bX$ as
\begin{equation}
  \label{eq:sumscore} S(\bx, Q) = \sum_k S_k(\bx_k, Q_k)
  \ ,
\end{equation}
where $\bX_k \sim Q_k$ when $\bX \sim Q$.  The form
(\ref{eq:sumscore}) localises the problem to the $\{\bX_k\}$, which
can simplify computation.

We term a scoring rule of the form \eqref{eq:sumscore} a {\em
  composite scoring rule\/}.  In the special case that each $S_k$ is
the log score, (\ref{eq:sumscore}) becomes a (negative log) {\em
  composite likelihood\/} (see, \eg, Varin {\em et al.\/}, 2011).
Composite likelihood is often considered as a surrogate for the full
likelihood function, useful in models with a complex dependence
structure.  The above reformulation allows us to treat composite
likelihood in its own right, as supplying a proper scoring rule.  And
from this point of view, as we shall see, there is nothing special
about composite likelihood: most of the existing results about it
extend with very little change to the more general case of an
arbitrary proper scoring rule (whether or not constructed as a
composite score).

\begin{example}
  \begin{rm}
    Consider a spatial process $\bX = (X_v : v \in V)$, where $V$ is a
    set of lattice sites.  For a joint distribution $Q$ for $\bX$, let
    $Q_v$ be the family of conditional distributions for $X_v$, given
    the values of $\bX_{\setminus v}$, the variables at all other
    sites.  If $Q$ is Markov, $Q_v$ depends only on $\bX_{\nei{v}}$
    (variables at sites neighbouring $v$).  We can then construct a
    proper scoring rule
    \begin{math}
      S(x, Q) = \sum_v S_0(x_v, Q_v),
    \end{math}
    where $S_0$ is a proper scoring rule for the state at a single
    site.  When $S_0$ is the log score this is the (negative log) {\em
      pseudo-likelihood\/} of Besag (1975).  For binary $X_v$ and
    $S_0$ the Brier score, it leads to the {\em ratio matching\/}
    method of Hyv\"arinen (2005).  Some comparisons may be found in
    Dawid and Musio (2013).
  \end{rm}
\end{example}



\section{Scoring rule inference}

Let ${\cal P} = \{P_{\theta}: \theta \in \Theta\}$, with $\Theta$ an
open subset of $\R^p$, be a parametric family of distributions on
$\cX$, and let $p_\theta(x)$ denote the probability density function
of $P_\theta$.  The validity of inference about $\theta$ using scoring
rules can be justified invoking the general theory of unbiased
estimating functions.

Consider a proper scoring rule $S$ on $\cX$, and write $S(x,\theta)$
for $S(x,P_\theta)$ and $s(x,\theta)$ for the gradient vector of
$S(x,\theta)$ with respect to $\theta$, that is,
\begin{eqnarray}
  s(x,\theta) = \nabla_\theta S(x,\theta) = \frac{\partial S(x, \theta)}{\partial \theta}
  \ .
  \label{ef}
\end{eqnarray}

For $X \sim P$, where $P$ might not belong to ${\cal P}$, we can
approximate $P$ within ${\cal P}$ by $P_{\theta_P}$, where
\begin{equation}
  \label{eq:approxd}
  \theta_P = \arg\min_\theta D(P,P_\theta),
\end{equation}
where $D$ is the discrepancy associated with $S$.  In particular, if
$P=P_{\theta_0}\in {\cal P}$, where $\theta_0$ is the true value of
the parameter, then $\theta_P = \theta_0$.  Since $D(P,P_\theta) =
S(P,P_\theta) - H(P)$, (\ref{eq:approxd}) is equivalent to
\begin{equation}
  \label{eq:approxs}
  \theta_P = \arg\min_\theta S(P,P_\theta).
\end{equation}

Now let $(x_{1},\ldots,x_{n})$ be a random sample of size $n$ from
$P$, and let $\widehat P_n$ be the associated empirical distribution.
Then we can take $\widehat{\theta}_{S} = \theta_{\widehat P_n}$ as a
point estimate of $\theta_P$: that is, $\widehat{\theta}_{S}$ is the
value of $\theta$ minimising $S(\widehat P_n, P_\theta)$.
Equivalently, it minimises $n S(\widehat P_n, P_\theta)$, which is
just the total {\em empirical score\/}
$$
S(\theta) = \sum_{i=1}^n S(x_{i}, \theta) \ .
$$
Thus the {\em scoring rule estimate\/} of $\theta_P$ is
$$
\widehat{\theta}_{S}=\arg\min_{\theta} S(\theta) = \arg\min_{\theta}
\sum_{i=1}^{n}S(x_i,\theta) \ ,
$$
which (under differentiability conditions) is the solution of the {\em
  scoring rule estimating equation\/}
\begin{equation}
  \label{eq:1}
  s(\theta) = \sum_{i=1}^{n}s(x_i,\theta) = 0.
\end{equation}
Note that when $S(\theta)$ is the log score, \ie\ $S(\theta)=-
\sum_{i=1}^n \log p_\theta(x_i)$, the scoring rule estimating equation
(\ref{eq:1}) is just the (negative of) the likelihood equation, and
the scoring rule estimate is just the maximum likelihood estimate.

For the special case that the discrepancy $D$ is the Tsallis/density
power divergence, Basu {\em et al.\/} (1998) note that---unlike many
other applications of minimum distance estimation (see for instance
Cao {\em et al.\/}, 1995)---this procedure does not require the
preliminary construction of a continuous nonparametric density
estimate of the true density $p(\cdot)$, so avoiding complications
such as bandwidth selection.  This pleasant property extends to all
minimum discrepancy estimates based on a proper scoring rule.

Generalising a familiar property of the likelihood equation, the
following theorem (see Dawid and Lauritzen, 2005; Dawid, 2007) shows
that, for any proper scoring rule and any family of distributions, the
scoring rule estimating equation \eqref{eq:1} is unbiased.

\begin{theorem}
  \label{teo:1}
  For the scoring rule estimating function $s(x,\theta)$, it holds
  that
  $$
  \E_P\, \left\{ s(X,\theta_P) \right\} =0 \ ,
  $$
  where $\E_P (\cdot)$ denotes expectation with respect to $P$.
\end{theorem}

\begin{proof}
  For fixed $P$, $\E_{P}\, S(X,\phi)$ is minimised at $\phi =
  \theta_P$.  Thus, under sufficient regularity to allow interchange
  of expectation over $X$ and differentiation with respect to
  $\theta$, we have
  \begin{eqnarray*}
    0 &=& \left.\nabla_\phi\, \E_P\,S(X,\phi)\right|_{\phi=\theta_P}\\
    &=&  \left.\E_P\,\nabla_\phi\,S(X,\phi)\right|_{\phi=\theta_P}\\
    &=& \E_P\,s(X,\theta_P).
  \end{eqnarray*}
\end{proof}

\begin{corollary}
  \label{cor:1}
  For $P=P_\theta \in {\cal P}$,
  $$
  \E_{\theta}\, \left\{s(X,\theta) \right\} =0 \ ,
  $$
  where $\E_\theta (\cdot)$ denotes expectation with respect to
  $P_\theta$.
\end{corollary}

As a consequence of Corollary~\ref{cor:1}, we have that equation
(\ref{eq:1}) delivers an unbiased estimating equation for the
parameter $\theta$, that is the first Bartlett identity holds.  The
solution thus forms a special case of $M$-estimation (see, among
others, Hampel {\em et al.\/}, 1986, and Huber and Ronchetti, 2009).
An important feature of this approach is that the choice of the
scoring rule is entirely independent of the specific estimation
problem under consideration.  Any such choice supplies a universal
$M$-estimation procedure, applying across all possible models in
mutually consistent fashion.  This thus extends the familar universal
applicability of maximum likelihood estimation to scoring rules other
than the log score.

\subsection{Example: Bregman estimation}

Consider the separable Bregman score given by \eqref{eq:sepbreg}.  We
have
\begin{equation}
  \label{eq:sbreg}
  -s(x,\theta) = \lambda(x,\theta) - \E_\theta \lambda(X,\theta)
\end{equation}
with
\begin{eqnarray}
  \label{eq:lambda}
  \lambda(x,\theta) &=& \nabla_\theta\psi'\{p_\theta(x)\}\\
  &=&  \psi''\{p_\theta(x)\}\nabla_\theta p_\theta(x).
\end{eqnarray}
Since the function $\psi$ was required to be convex, we have that
$\alpha := \psi''$ must be non-negative.  Any such choice of $\alpha$
determines a suitable function $\psi$, and hence a separable Bregman
scoring rule.  We term such a choice for $\alpha$ a {\em Bregman
  gauge\/}.

Having fixed on a Bregman gauge function $\alpha$, we can now solve
any estimation problem, of any parametric dimensionality, based on
observations on $X$, by using the estimating function
\begin{equation}
  \label{eq:univest}
  \lambda(x,\theta) = \alpha\{p_\theta(x)\}\nabla_\theta p_\theta(x).
\end{equation}
An unbiased estimating equation for $\theta$, yielding an
$M$-estimator, is obtained by equating the sample and population
averages of $\lambda$.  The form \eqref{eq:univest} is, in this sense,
a universal estimating function.  For the special Bregman gauge
$\alpha(t) \equiv 1/t$ we recover Fisher's efficient score function
and maximum likelihood estimation.

\subsubsection{Location model}
\label{sec:bregloc}

Bregman inference for a location model model is particularly
straightforward.  For such a model we have
\begin{equation}
  \label{eq:locmod}
  p_\theta(x) = f(x-\theta),
\end{equation}
where $f$ is a density on $\R$ that we assume to be strictly positive
everywhere and continuously differentiable.  Using the separable
Bregman score formula \eqref{eq:sepbreg}, we note that the integral
term in $S(x,\theta)$ does not depend on $\theta$.  Consequently, for
the case of a location model, minimising the empirical score is
equivalent to maximising
\begin{equation}
  \label{eq:sumf}
  \sum_{i=1}^n \xi\{f(x_i-\theta)\},
\end{equation}
where $\xi = \psi'$ is a fixed increasing function (and $\xi'$ is just
the Bregman gauge $\alpha$).  This generalises maximum likelihood, for
which $\xi \equiv \ln$.

The maximum of \eqref{eq:sumf} will be obtained by setting its
derivative to $0$, leading to the unbiased estimating equation
\begin{displaymath}
  \sum_{i=1}^n \lambda(x_i,\theta) = 0,
\end{displaymath}
where, in accordance with \eqref{eq:univest},
\begin{displaymath}
  \lambda(x,\theta) = -\alpha\left\{f(x-\theta)\right\} f'(x-\theta).
\end{displaymath}
In this case, $\E_\theta \lambda(X,\theta)$ is identically $0$.

\section{Asymptotics}
\label{sec:asym}
Given a proper scoring rule $S$, we can apply standard results on
$M$-estimators to describe the properties of the scoring rule
estimator $\widehat{\theta}_{S}$ defined by \eqref{eq:1}.
Hereinafter, regularity conditions as detailed in \eg
Barndorff-Nielsen and Cox (1994, Section 9.2) or in Molenberghs and
Verbeke (2005, Sec.~9.2.2), are assumed.

\begin{theorem}
  \label{teo:2}
  Under suitable regularity conditions, the scoring rule estimator
  $\widehat{\theta}_{S}$ is consistent and asymptotically normal, with
  mean $\theta_P$ and variance $V$, where
$$
V = K^{-1} J (K^{-1})^T\ ,
$$
with
\begin{eqnarray}
  \label{eq:J}
  J &=& \E_{P}\left\{s(\theta_P)s(\theta_P)^T\right\}\\
  \label{eq:H}
  K &=& \E_{P}\left.\left\{\frac{\partial s(\theta)}
      {\partial \theta^T}\right\}\right|_{\theta=\theta_P}.
\end{eqnarray}
When $P=P_\theta$, then $V=V(\theta)=K(\theta)^{-1} J(\theta)
(K(\theta)^{-1})^T$, with $J(\theta)=\E_{\theta}
\left\{s(\theta)s(\theta)^T\right\}$ and $K(\theta)=\E_{\theta}
\left\{\frac{\partial s(\theta)} {\partial \theta^T}\right\}$.
\end{theorem}

The matrix $G = V^{-1} $ is known as the Godambe information matrix
(Godambe, 1960).  The form of $V$ is due to the failure of the second
Bartlett identity since, in general, $K \neq J$.  In the special case
of the log score, \ie\ when $S(\theta)=- \sum_{i=1}^n \log
p_\theta(x_i)$, and for $P=P_\theta$, we have that $G =
K(\theta)=J(\theta)$ is the Fisher information matrix.


\subsection{Scoring rule test statistics}

Hypothesis testing and confidence regions for $\theta$ can be formed
in the usual way by using a consistent estimate of the asymptotic
variance $V$.  In particular, inference for $\theta$ can be based on
the scoring rule Wald-type statistic

\begin{eqnarray}
  W^S_w(\theta) = (\widehat\theta_S - \theta)^T V^{-1} (\widehat\theta_S - \theta)
  \ ,
  \label{swe}
\end{eqnarray}
which has an asymptotic chi-squared on $p$ degrees of freedom
distribution.  The asymptotic $\chi^2_p$ distributional result holds
also for the scoring rule score-type statistic $W^S_s(\theta) =
s(\theta)^T J^{-1} s(\theta)$.  A consistent estimate of $V$ can be
obtained using estimates of the matrices $J$ and $K$:
$$
\widehat J = \sum_{i=1}^n
s(x_i,\widehat{\theta}_{S})s(x_i,\widehat{\theta}_{S})^T \quad \quad
\widehat K = \sum_{i=1}^n \left.  \partial s(x_i,\theta)/\partial
  \theta^T \right|_{\theta=\widehat{\theta}_{S}} \ ;
$$
one can refer to Varin (2008) and Varin {\em et al.\/} (2011) for a
detailed discussion of the issues related to the estimation of $J$ and
$K$.

As is well known, Wald-type statistics lack invariance under
reparameterisation, and force confidence regions to have an elliptical
shape.  On the other hand, score-type statistics are seen to suffer
from numerical instability in many examples (see, \eg, Molenberghs and
Verbeke, 2005, Chap.\ 9).  In this respect a scoring rule ratio
statistic, of the form
\begin{eqnarray}
  W^S(\theta) = 2 \left\{ S(\theta) - S(\widehat\theta_S) \right\}
  \ ,
  \label{sw}
\end{eqnarray}
seems to be a more appealing basis for inference.  However, the
asymptotic distribution of (\ref{sw}) departs from the familiar
likelihood result, and involves a linear combination of independent
chi-square random variables with coefficients given by the eigenvalues
of a matrix related to Godambe information (see, among others,
Heritier and Ronchetti, 2004, and Varin {\em et al.\/}, 2011).  More
precisely,
$$
W^S(\theta) \stackrel{L}{\to} \sum_{j=1}^p \mu_j Z_j^2 \ ,
$$
where $\mu_1,\ldots, \mu_p$ are the eigenvalues of $J K^{-1} = K
G^{-1}$ and \linebreak $Z_1, \ldots,Z_p$ are independent standard
normal variates.

Analogous limiting results can be shown to hold for tests on subsets
of $\theta$.  Let $\theta$ be partitioned as $\theta=(\psi, \lambda)$,
where $\psi$ is a $p_0$-dimensional parameter of interest and
$\lambda$ is a $(p-p_0)$-dimensional nuisance parameter.  With this
partition, the scoring rule estimating function is similarly
partitioned as $s(\theta)=(s_\psi(\theta),s_\lambda(\theta))$, where
$s_\psi(\theta)=(\partial/\partial \psi) S(\theta)$ and
$s_\lambda(\theta)=(\partial/\partial \lambda) S(\theta)$.  Moreover,
consider the further partitions
$$
K = \left[
  \begin{array}{cc}
    K_{\psi \psi} & K_{\psi \lambda} \\
    K_{\lambda \psi} & K_{\lambda \lambda}
  \end{array}
\right] \ , \quad K^{-1} = \left[
  \begin{array}{cc}
    K^{\psi \psi} & K^{\psi \lambda} \\
    K^{\lambda \psi} & K^{\lambda \lambda}
  \end{array}
\right] \ ,
$$
and similarly for $G$ and $G^{-1}$.  Finally, let
$\widehat\theta_{S\psi}$ be the constrained scoring rule estimate of
$\theta$ for fixed $\psi$, and let $\widehat\psi_S$ be the $\psi$
component of $\widehat\theta_{S\psi}$.

A profile scoring rule Wald-type statistic for the $\psi$ component
may be defined as
$$
W^S_{wp} (\psi) = (\widehat\psi_S - \psi)^T (G^{\psi \psi})^{-1}
(\widehat\psi_S - \psi) \ ,
$$
and it has an asymptotic $\chi^2_{p_0}$ null distribution.  Moreover,
using the asymptotic result (Rotnitzky and Jewell, 1990)
$s_\psi(\widehat\theta_{S\psi}) \, \dot{\sim} \, N_{p_0} \left( 0,
  (K^{\psi \psi})^{-1} G^{\psi \psi} (K^{\psi \psi})^{-1} \right)$,
the profile scoring rule score-type statistic $W^S_{sp} (\psi) = s_\psi
(\widehat\theta_{S\psi})^T K^{\psi \psi} (G^{\psi \psi})^{-1} K^{\psi
  \psi} s_\psi (\widehat\theta_{S\psi})$ has an asymptotic
$\chi^2_{p_0}$ null distribution.  Finally, we have that the
asymptotic distribution of the profile scoring rule ratio statistic
for $\psi$, given by
$$
W^S_p(\psi) = 2 \left\{S(\widehat\theta_{S\psi}) - S(\widehat\theta_S)
\right\} \ ,
$$
is $\sum_{j=1}^{p_0} \nu_j Z_j^2$, where $\nu_1,\ldots,\nu_{p_0}$ are
the eigenvalues of $(K^{\psi \psi})^{-1} G^{\psi \psi}$.  This result
follows from Kent (1982, Theorem 3.1).  When evaluating the
eigenvalues of $(K^{\psi \psi})^{-1} G^{\psi \psi}$ it is possible to
replace $\theta$ with $\widehat\theta_{S\psi}$.


\subsection{Calibration of the scoring rule ratio statistic}

Since the asymptotic null distribution of scoring rule ratio
statistics depends both on the statistical model and on the parameter
of interest, adjustments to $W^S(\theta)$ and $W^S_p (\psi)$ are of
interest.  These adjustments aim for an asymptotic null distribution
that depends only on the dimension of the parameter of interest, and
they have been discussed in the statistical literature for general
pseudo-likelihood functions based on unbiased estimating equations;
see, among others, Varin (2008), Pace {\em et al.\/} (2011, 2013),
Varin {\em et al.\/} (2011), and references therein.

First, let us consider the scalar parameter case.

\begin{theorem}
  \label{thm:4.1}
  For $p=1$, the adjusted scoring rule ratio statistic satisfies
  \begin{eqnarray}
    W^S(\theta)_{adj} = \frac{W^S(\theta)}{\mu_1}  \stackrel{L}{\to} \chi^2_1
    \ ,
    \label{swadj}
  \end{eqnarray}
  where $\mu_1=J/K$.
\end{theorem}

The proof of Theorem \ref{thm:4.1} is based of the well-known results
in Heritier and Ronchetti (1994) and Pace {\em et al.} (2011).

For $p > 1$, simple adjustments of the form (\ref{swadj}) for
$W^S(\theta)$ based on moment conditions can be considered as well.
For instance, first-order moment matching (see, \eg, Rotnitzky and
Jewell, 1990, Molenberghs and Verbeke, 2005, Sec.~9.3.3) gives the
adjustment
\begin{eqnarray}
  W^S(\theta)_{m1} = \frac{W^S(\theta)}{\bar\mu} \ ,
  \label{m1adj}
\end{eqnarray}
where $\bar\mu = \sum_{i=1}^p \mu_i/p = \mbox{tr}(JK^{-1})/p$.  A
$\chi^2_p$ approximation is used for the null distribution of
$W^S(\theta)_{m1}$.  Matching of moments up to higher order can also be
considered, as in Satterthwaite (1946) and Wood (1989); see also
Lindsay {\em et al.} (2000).
Note, however, that the correction (\ref{m1adj}) to $W^S(\theta)$ might
be inaccurate because it corrects only the first moment of the
distribution and it does not recover the usual $\chi^2_p$ asymptotic
distribution.

For $p>1$, calibration of $W^S(\theta)$ can be based on the following
theorem.

\begin{theorem}
  \label{thm:4.2}
  Using the rescaling factor
  \begin{eqnarray}
    A(\theta) =   \frac{s(\theta)^T J^{-1}
      s(\theta)}{s(\theta)^T K^{-1} s(\theta)}
    \ ,
    \label{at}
  \end{eqnarray}
  we have
  \begin{eqnarray}
    W^S(\theta)_{inv} = A(\theta) \, W^S(\theta)  \stackrel{L}{\to} \chi^2_p
    \ .
    \label{swinv}
  \end{eqnarray}
\end{theorem}

The proof of Theorem \ref{thm:4.2} is based on formulae in Pace {\em
  et al.} (2011, 2013), who discuss alternatives to moment-based
adjustments for likelihood-type ratio statistics, aiming to obtain a
statistic with the usual $\chi^2_p$ asymptotic distribution.

In the situation with nuisance parameters, adjustments of the form
$W^S_p (\psi)_{m1}$ and $W^S_p (\psi)_{m2}$ to $W^S_p (\psi)$, analogous
to $W^S(\theta)_{m1}$ and $W^S(\theta)_{m2}$, respectively, can be
easily defined using the eigenvalues $\nu_1,\ldots,\nu_{p_0}$ of
$(K^{\psi \psi})^{-1} G^{\psi \psi}$, evaluated at
$\widehat\theta_{S\psi}$.  The extension of $W^S(\theta)_{inv}$ in the
nuisance parameter case can be obtained following the results in Pace
{\em et al.\/} (2011).  We obtain
$$
W^S_p(\psi)_{inv} =
\frac{W^S_{sp}(\psi)}{s_\psi(\widehat\theta_{S\psi})^T K^{\psi \psi}
  (\widehat\theta_{S\psi}) s_\psi(\widehat\theta_{S\psi})} \,
W^S_p(\psi) \ .
$$

\section{Robustness}
\label{sec:IF}
The influence function (\IF) (see, \eg, Hampel {\em et al.}, 1986, Chap.\ 2) of an estimator
measures the effect on it of a small contamination at the point $x$,
standardized by the mass of that contamination.  The supremum of the
$\IF$ over the data-space measures the worst influence of such
contamination, so supplying a measure of gross-error sensitivity.  A
desirable robustness property for a statistical procedure is that the
gross-error sensitivity be finite, \ie, that the \IF\ be bounded.
This is termed {\em $B$-robustness\/}.

From the general theory of $M$-estimators (see, \eg, Huber and
Ronchetti, 2009), the \IF\ of the estimator $\widehat{\theta}_{S}$,
the solution of the unbiased estimating equation \eqref{eq:1}, is
given by
\begin{equation}
  \label{eq:IF}
  \IF(x;s,P) = K^{-1} s(x,\theta_P).
\end{equation}
Thus, if the function $s(x,\theta)$ is, for each $\theta$, bounded in
$x$, then the corresponding scoring rule estimator
$\widehat{\theta}_{S}$ is $B$-robust.  Note that, in general, the form
of the function $s(x,\theta)$ depends on the model ${\cal P}$ as well
as the scoring rule $S$.  Finally, notice that the \IF\ can also be
used to evaluate the asymptotic variance of $\widehat{\theta}_{S}$,
since
$$
V = E_P \left\{ \IF(X; s,P) \, \IF(X;s,P)^T \right\}\ .
$$


\subsection{Example: robustness of Bregman estimate}
\label{sec:bregrob}

A necessary and sufficient condition for $B$-robustness of the Bregman
estimate, where $s$ is given by \eqref{eq:sbreg} with $\lambda$
determined by \eqref{eq:univest}, is:
\begin{cond}
  \label{cond:nsc}
  For all $\theta$, $\lambda(x,\theta)$ is a bounded function of $x$.
\end{cond}

The above condition inextricably combines properties of the Bregman
gauge function $\alpha$ and the form of the model $p_\theta$.  We can
also identify a useful set of sufficient conditions for $B$-robustness,
which handles these ingredients separately.  First we introduce a
definition.
\begin{definition}
  \label{def:lb} We say that a function $f:\R^+\rightarrow\R^+ $ is
  {\em locally bounded\/} if $f(t)$ is bounded on each finite interval
  $0 < t < M$.
\end{definition}
In this case, $f(0) = \lim_{t \downarrow 0} f(t)$ (if it exists) must
be finite.  For our applications, this condition will typically be
sufficient.

\renewcommand{\theenumi}{(\roman{enumi})}

It now follows that a sufficient condition for $B$-robustness of the
Bregman estimate is:
\begin{cond}
  \label{cond:gsc}
  \quad\vspace{-4ex}\\
  \begin{enumerate}
  \item \label{it:psi} The Bregman gauge $\alpha = \psi''$ is locally
    bounded, and
  \item\label{it:p} both $p_\theta(x)$ and $\nabla_\theta p_\theta(x)$
    are bounded in $x$, for each $\theta$.
  \end{enumerate}
\end{cond}
Note that if \condref{nsc} or \condref{gsc}~\itref{p} hold for one
parametrisation, they equally hold for any other.

The Brier score, with $\psi(t)=t^2$, satisfies
\condref{gsc}~\itref{psi}---indeed, $\alpha(t) \equiv 2$ is bounded on
the whole of $(0,\infty)$.  Other such ``totally bounded'' examples
include $\psi(t)= 2 t \tan^{-1}(t) - \ln(1 + t^2)$, with
$\alpha(t)=2/(1+t^2)$, and $\psi(t)= (1+t)\ln(1 + t)$, with
$\alpha(t)=1/(1+t)$.  The Tsallis/density power score, with $\psi(t)
\propto t^\gamma$ and $\alpha(t) \propto t^{\gamma-2}$ is locally
bounded but not totally bounded for $\gamma> 2$.  However for the log
score, with $\psi(t) \equiv t\ln(t)$, $\alpha(t) \equiv 1/t$ is not
bounded at $0$, so this particular Bregman scoring rule violates the
local boundedness \condref{gsc}~\itref{psi}.  And this is reflected in
the fact that the maximum likelihood estimator is typically not
$B$-robust.

For a real location model, with $p_\theta(x) = f(x-\theta)$, the
Bregman score will yield a $B$-robust estimator if and only if
\begin{cond}
  \label{cond:nscloc}
  $(d/du) \psi'\{(f(u)\}=\psi''\{f(u)\} f'(u)$ is bounded.
\end{cond}
In particular (cf.\ Basu {\em et al.\/}, 1998), for a real location
model the necessary and sufficient condition that the Tsallis/density
power score supply
  a
$B$-robust estimator is that $f(u)^{\gamma-2} f'(u)$ be a bounded
function of $u$.

A {\em sufficient\/} condition for \condref{nsc} to hold is:
\begin{cond}
  \label{cond:sc}
  \quad\vspace{-4ex}\\
  \begin{enumerate}
  \item $\alpha$ is locally bounded
  \item\label{it:f} $f'(u)$ is bounded.
  \end{enumerate}
\end{cond}
\condref{sc}~\itref{f} implies \condref{gsc}~\itref{p}, since
boundedness of $f'$ implies boundedness of $f$ (see \lemref{bounded}
in \appref{bounded}).  For instance, \condref{sc}~\itref{f} holds for
the normal, logistic, Cauchy and extreme value distributions.


For a real scale model, with $p_\theta(x)=\theta f(\theta x)$
($x,\theta >0$), the Bregman score yields a $B$-robust estimator if and
only if
\begin{cond}
  \label{cond:scale}
  $\alpha\{\theta f(\theta x)\}\{f(\theta x) + \theta x f'(\theta
  x)\}$ is bounded in $x$ for all $\theta$.
\end{cond}

We have the following sufficient condition:
\begin{cond}
  \label{cond:scalebis}
  \quad\vspace{-4ex}\\
  \begin{enumerate}
  \item \label{it:alpha} $\alpha$ is locally bounded
  \item \label{it:beta} $f(u)$ and $u f'(u)$ are bounded on $\R^+$.
  \end{enumerate}

\end{cond}

We again remark that \condref{scalebis}~\itref{alpha} holds for the
Brier and Tsallis score, but not for the log score.  The log normal,
exponential, and Gamma (with $\alpha \geq 1$) densities satisfy
\condref{scalebis}~\itref{beta}.  For a general location-scale model,
and more generally for a regression-scale model, a sufficient
condition for \condref{nsc} to hold is: $(i)$ $\alpha$ is locally
bounded, and $(ii)$ $f(u)$, $f'(u)$ and $uf'(u)$ are bounded on
$\R^+$.

\section{Examples}
\label{sec:examples} In this section we provide simulation results to
assess coverage probabilities of confidence regions based on the
adjustments of the scoring rule ratio statistic $W^S(\theta)$.  Three
examples are described.  The first deals with a multivariate normal
distribution, the second with a location-scale model, and the third
with a linear regression model.  The examples are chosen so that we
can easily do closed form calculations for both the Tsallis score
(\ref{eq:tsallisscore}) and the log score.  In the two last examples
the focus is on showing the accuracy of the calibration of the scoring
rule ratio statistic, and on studying the robustness properties of the
Tsallis score with respect to classical robust procedures based on
$M$-estimators.


\begin{example}{\bf Equi-correlated normal model}
  \label{ex:Example1}

  \begin{rm}
    We discuss inference on the correlation coefficient $\rho$ of an
    equi-correlated multivariate normal distribution.  This
    illustrative example is considered by Cox and Reid
    (2004).

    Let $(X_i: i=1, \ldots,n)$ be independent realizations of a
    $q$-variate normal random variable, with standard margins and with
    corr$(X_{ir},X_{is})=\rho$ ($r,s=1,\ldots,q$, $r \neq s$).  Thus
    the density function of $X_i$ is
$$
p(x_i;\rho) = \frac{\exp \left\{ - \frac{1}{2(1-\rho)} \left(
      \sum_{r=1}^q x_{ir}^2 - \frac{\rho q^2}{1-\rho(q-1)} \overline
      x_i^2 \right) \right\}}{ \sqrt{ (2 \pi)^q (1-\rho)^{(q-1)}
    \{1+\rho(q-1)\}}} \ ,
$$
where $\overline x_i := \sum_{r=1}^q x_{ir}/q$.

Straightforward calculations show that the Tsallis empirical score is
$S(\rho)=\sum_{i=1}^n S(x_i,\rho)$, with
$$
S(x_i,\rho) = -\gamma p(x_i;\rho)^{(\gamma-1)} + \frac{(\gamma-1)
}{\sqrt{\gamma^q (2 \pi)^{q(\gamma-1)} (1-\rho)^{(\gamma-1)(q-1)}
    \{1+\rho(q-1)\}^{(\gamma-1)}}} \ .
$$
                            
In order to assess the quality of the proposed adjustment
$W^S(\rho)_{adj}$ (see \thmref{4.1}) of the scoring rule ratio
statistic based on $S(\rho)$, we ran a simulation experiment with
$n=30$, $q=10$ and $\rho=0.5$.  For comparison we also consider the
pairwise log-likelihood, given by
$$
\ell^P (\rho) = -\frac{nq(q-1)}{4} \log(1-\rho^2) - \frac{q - 1 +
  \rho}{2(1-\rho^2)} SS_W - \frac{(q - 1)(1 - \rho)}{2(1-\rho^2)}
\frac{SS_B}{q} \ ,
$$
where $SS_W = \sum_{i=1}^n \sum_{r=1}^q (x_{ir}-\overline x_i)^2$ and
$SS_B = q^2 \sum_{i=1}^n \overline x_i^2$: see Cox and Reid (2004),
Pace {\em et al.\/} (2011), who find that the adjustment of the
pairwise likelihood ratio statistics has reasonable coverage
properties.  Note that the pairwise log-likelihood, as an example of
composite log-likelihood, is a special case of a proper scoring rule.

Table~\ref{tabrho} reports the empirical coverages of confidence
intervals based on several statistics: the full likelihood ratio
$W(\rho)$, the Wald statistic from the full model $W_w(\rho)$, the
Tsallis Wald statistic $W^S_w(\rho)$ and the adjustment
(\ref{swadj}) of the Tsallis empirical score likelihood ratio
statistic $W^S(\rho)_{adj}$ for three values of $\lambda$.
Finally, also the pairwise Wald statistic $W^P_w(\rho)$ and the
adjustment (\ref{swadj}) of the pairwise likelihood ratio statistic
$W^P(\rho)_{adj}$ are given.  We note that the proposed adjustment
(\ref{swadj}) of $W^S(\rho)$ shows a reasonable performance in terms of
coverage.  In particular, when $\lambda$ is small, it proves to be a
good competitor of the pairwise likelihood ratio statistic
$W^P(\rho)_{adj}$, with the advantage of using the full likelihood.
However the Tsallis Wald statistic $W^S_w(\rho)$ appears
useless.

\begin{table}
  \begin{center}
    \begin{tabular}{|l|lll|lll|lll|} \hline $1-\alpha$ & 0.90 & 0.95 &
      0.99 \\ \hline
      $W(\rho)$                          & 0.903 & 0.942 &  0.993 \\
      $W_w(\rho)$                     &  0.903 & 0.939 & 0.994  \\
      $W^S_w(\rho)$, $\lambda=2$               & 1.000 & 1.000 & 1.000 \\
      $W^S_w(\rho)$, $\lambda=1.5$         &1.000 & 1.000 & 1.000  \\
      $W^S_w(\rho)$, $\lambda=1.25$      & 1.000 & 1.000 & 1.000 \\
      $W^S(\rho)_{adj}$, $\lambda=2$         & 0.789 & 0.830 & 0.883  \\
      $W^S(\rho)_{adj}$, $\lambda=1.5$   & 0.849 & 0.906 & 0.958  \\
      $W^S(\rho)_{adj}$, $\lambda=1.25$ &0.886 & 0.937 & 0.982  \\
      $W^P(\rho)_{adj}$             & 0.895 & 0.945 &0.991  \\
      $W^P_w(\rho)$ & 0.892 & 0.938 & 0.989 \\ \hline
    \end{tabular}
    \caption{Equicorrelated multivariate normal model.  Empirical
      coverage of $(1-\alpha)$ confidence intervals based on different
      statistics, based on 5.000 replications, with $n = 30$, $q =
      10$, $\rho = 0.5$, and $\lambda=2, 1.5, 1.25$.  }
    \label{tabrho}
  \end{center}
\end{table}    
\end{rm}
\end{example}


\begin{example}{\bf Scale and location model}
  \label{ex:Example2}

    \begin{rm}
      Let $\theta=(\mu,\sigma)$, where $\mu \in \Real$ is a location
      parameter and $\sigma>0$ a scale parameter.  In this case we
      have $p (x;\theta)=p_0\{(x-\mu)/\sigma\}/\sigma$, where
      $p_0(\cdot)$ is the standard distribution.  The Tsallis
      empirical score is $S(\theta)=\sum_{i=1}^n S(x_i,\theta)$, with
$$
S(x_i,\theta) = -\gamma \, p(x_i;\theta)^{(\gamma-1)} +
\frac{(\gamma-1) }{\sigma^{(\gamma-1)}} \int p_0(x)^\gamma \, dx \ ,
$$
for $i=1,\ldots,n$.

We ran a simulation experiment, for several values of $n$ and with
$\lambda=2, 1.5, 1.25$, in order to assess the quality of the proposed
adjustments of the Tsallis scoring rule ratio statistic based on
$S(\theta)$.  For comparison, we considered also the well-known Huber
location-scale $M$-estimator (see Hampel {\em et al.\/}, 1986,
Sec.~4.2).  For this estimator, only the Wald type statistic
$W^H_w(\theta)$ is available.

Table \ref{tabSL} gives the results of a Monte Carlo experiment that
compares confidence regions for $\theta$ based on the full likelihood
ratio $W(\theta)$, the Tsallis Wald statistic $W^S_w(\theta)$
and the adjustments (\ref{swadj}) and (\ref{m1adj}) of the Tsallis
empirical score likelihood ratio statistic, and the Huber Wald
statistic $W^H_w(\theta)$, when the central model is the normal one.
Data are generated from two different distributions: the $N(0,1)$
model, and the contaminated model $0.95 \cdot N(0,1) + 0.05 \cdot
N(0,10^2)$.  We note that the proposed adjustments of
$W^S(\theta)$ show a reasonable performance in terms of
coverage, both under the central model and under the contaminated
model.  However, the Tsallis Wald statistic $W^S_w(\theta)$
and the Huber Wald statistic $W^H_w(\theta)$ exhibit poor coverage under
the contaminated model.

\begin{table}
  \begin{center}
    \begin{tabular}{|l|ccc|ccc|} \hline & & $N(0,1)$ & & &
      cont.~$N(0,1)$ & \\ \hline & $n=10$ & $n=20$ & $n=30$ & $n=10$ &
      $n=20$ & $n=30$ \\ \hline
      $W(\theta)$                          & 0.934 & 0.938 &  0.942 & 0.652 & 0.475 &  0.357  \\
      $W^S_w(\theta)$, $\lambda=2$               & 0.914 & 0.926 & 0.937  & 0.913 & 0.926 & 0.931  \\
      $W^S_w(\theta)$, $\lambda=1.5$         & 0.928 & 0.939 & 0.942  & 0.914 & 0.924 & 0.926  \\
      $W^S_w(\theta)$, $\lambda=1.25$      & 0.948 & 0.947 & 0.945  & 0.898 & 0.908 & 0.908  \\
      $W^S(\theta)_{inv}$, $\lambda=2$         & 0.872 & 0.918 & 0.931 & 0.886 & 0.931 & 0.939   \\
      $W^S(\theta)_{inv}$, $\lambda=1.5$   & 0.912 & 0.936 & 0.942 & 0.914 & 0.937 & 0.937  \\
      $W^S(\theta)_{inv}$, $\lambda=1.25$ &0.925 & 0.940 & 0.942  & 0.916 & 0.938 & 0.935 \\
      $W^S(\theta)_{m1}$, $\lambda=2$         & 0.981 & 0.967 & 0.962  & 0.978 & 0.959 & 0.952  \\
      $W^S(\theta)_{m1}$, $\lambda=1.5$   & 0.954 & 0.953 & 0.953 & 0.948 & 0.945 & 0.944  \\
      $W^S(\theta)_{m1}$, $\lambda=1.25$ &0.942 & 0.947 & 0.943 &0.925 & 0.937 & 0.934  \\
      $W^H_w(\theta)$ & 0.966 & 0.953 & 0.954 & 0.925 & 0.912 & 0.915 \\
      \hline
    \end{tabular}
    \caption{Scale and location model.  Empirical coverages (based on
      5000 replications) of 0.95 confidence regions based on different
      statistics, under the $N(0,1)$ model and the $0.95 \cdot N(0,1)
      + 0.05 \cdot N(0,10^2)$ contaminated model, with $\lambda=2,
      1.5, 1.25$.  }
    \label{tabSL}
  \end{center}
\end{table}
\end{rm}
\end{example}


\begin{example}{\bf Linear regression model}
  \label{ex:Example3}

    \begin{rm}
      Consider the linear regression model
      \begin{equation}
        \label{eq:linreg}
        y = X \beta + \sigma \varepsilon \ ,  
      \end{equation}
      where $X$ is a fixed $n \times p$ matrix, $\beta \in \Real^p$
      $(p \geq 1)$ an unknown regression coefficient, $\sigma>0$ a
      scale parameter and $\varepsilon$ an $n$-dimensional vector of
      random errors from a standard normal distribution.  We take
      $\sigma=1$ as known.  The Tsallis empirical score is
      $S(\beta)=\sum_{i=1}^n S(y_i,\beta)$, with
$$
S(y_i,\beta) = - \frac{\gamma}{(\sqrt{2 \pi})^{\gamma -1}} \exp
\left\{ - \frac{\gamma-1}{2} (y_i-x_i^{\T} \beta)^2 \right\}
+(\gamma-1) \int \phi(x)^\gamma \, dx \ ,
$$
where $x_i^{\T}$ is the $i$-th row of $X$ and $\phi(\cdot)$ is the
standard normal density.

In order to assess the quality of the proposed adjustments of the
Tsallis scoring rule ratio statistic based on $S(\beta)$, we ran a
simulation experiment with $p=3$ and for several values of $n$, with
$\lambda=2, 1.5, 1.25$.  For comparison, we considered also the
well-known Huber regression $M$-estimator (see Hampel {\em et al.\/},
1986).  As in the previous example, for this estimator only the Wald
type statistic $W^H_w(\beta)$ is available.

Our specific model is as follows.  In \eqref{eq:linreg}, all entries of
the first column of $X$ are $1$, those of the second column are
generated as independent standard normal variables, $z_1,\ldots,z_n$,
while the third column consists of the integers from $1$ to $n$.  The
model is $y_i = \beta_1 +\beta_2 z_i +\beta_3 i + \varepsilon_i$, and
the true parameter is $\beta=(1,2,3)$.  As for \exref{Example2},
$\varepsilon_1,\ldots,\varepsilon_n$ were generated from one of two
distributions: the $N(0,1)$ model, or the contaminated model $0.95
\cdot N(0,1) + 0.05 \cdot N(0,10^2)$.

Table~\ref{tabreg} compares confidence regions for $\beta$ based on
the full likelihood ratio $W(\beta)$, the Tsallis Wald statistic
$W^S_w(\beta)$ and the adjustments (\ref{swadj}) and
(\ref{m1adj}) of the Tsallis empirical score likelihood ratio
statistic, and the Huber Wald statistic $W^H_w(\beta)$, when the central
model is the normal one.  We note that the proposed adjustments of
$W^S(\theta)$ show a satisfactory performance in terms of
coverage, in particular when $\lambda$ is small, both under the
central model and under the contaminated model.  However the Tsallis
Wald statistic $W^S_w(\rho)$ and the Huber Wald statistic
$W^H_w(\theta)$ have poor coverage under the contaminated model.

\begin{table}
  \begin{center}
    \begin{tabular}{|l|ccc|ccc|} \hline & & $N(0,1)$ & & &
      cont.~$N(0,1)$ & \\ \hline & $n=15$ & $n=30$ & $n=50$ & $n=15$ &
      $n=30$ & $n=50$ \\ \hline
      $W(\beta)$                          & 0.950 & 0.954 &  0.951           & 0.605 & 0.518 &  0.417  \\
      $W^S_w(\beta)$, $\lambda=2$               & 1.000 & 1.000 & 1.000           & 1.000 & 1.000 & 1.000  \\
      $W^S_w(\beta)$, $\lambda=1.5$         & 0.999 & 0.998 & 0.998            & 0.988 & 0.998 & 1.000  \\
      $W^S_w(\beta)$, $\lambda=1.25$      & 0.966 & 0.927 & 0.975             & 0.884 & 0.948 & 0.976  \\
      $W^S(\beta)_{inv}$, $\lambda=2$         & 0.958 & 0.955 & 0.955            & 0.963 & 0.943 & 0.940   \\
      $W^S(\beta)_{inv}$, $\lambda=1.5$   & 0.949 & 0.954 & 0.952             & 0.948 & 0.948 & 0.939  \\
      $W^S(\beta)_{inv}$, $\lambda=1.25$ &0.949& 0.952 & 0.952             & 0.940 & 0.941 & 0.934 \\
      $W^S(\beta)_{m1}$, $\lambda=2$         & 0.958 & 0.955 & 0.955           & 0.963 & 0.943 & 0.940  \\
      $W^S(\beta)_{m1}$, $\lambda=1.5$   & 0.949 & 0.954 & 0.952           & 0.948 & 0.948 & 0.940  \\
      $W^S(\beta)_{m1}$, $\lambda=1.25$ &0.949 & 0.952 & 0.952            &0.940 & 0.941 & 0.934  \\
      $W^H_w(\beta)$ & 0.944 & 0.954 & 0.952 & 0.876 & 0.910 & 0.899 \\
      \hline
    \end{tabular}
    \caption{Linear regression model.  Empirical coverages (based on
      5000 replications) of 0.95 confidence regions based on different
      statistics, under the $N(0,1)$ and the $0.95 \cdot N(0,1) + 0.05
      \cdot N(0,10^2)$ models, with $\lambda=2, 1.5, 1.25$.  }
    \label{tabreg}
  \end{center}
\end{table}
\end{rm}
\end{example}

\section{Concluding remarks}
We have presented a general approach to parametric estimation theory,
based on replacing the full log-likelihood by a proper scoring rule.
This includes well-studied cases such as full, pseudo, composite,
pairwise \ldots log-likelihoods, as well a very wide variety of other
cases, not directly or indirectly related to likelihood at all.  Under
smoothness conditions, any proper scoring rule can be applied to any
statistical model, and delivers an associated $M$-estimator.  While
this may lose efficiency in comparison with full likelihood methods,
it can exhibit improved robustness or computational advantages.  In
\secref{IF} we identified some common situations where use of an
appropriate scoring rule achieves $B$-robustness.

We can use a scoring-rule estimator to construct hypothesis tests and
confidence intervals.  In addition to obtaining analogues of the Wald
and score test statistics, which are available for general
$M$-estimators, when basing inference on a scoring rule we also have
an analogue of the Wilks (log-likelihood ratio) statistic.  The
distributions of these analogues differ from those based on the full
likelihood, and we have considered adjustments to bring them more into
line.  The simulation studies in \secref{examples} indicate that
adjusted scoring rule likelihood-ratio type statistics yield
confidence regions whose coverage properties are satisfactory.  Both
the moment-matching correction and the correction given in
\thmref{4.2} perform well, and are preferred to the use of Wald type
statistics.

In more realistic applications, analytic expressions for the required
terms $K$ and $J$ may be unavailable, and numerical evaluation would
then seem to offer the most straightforward solution.  This issue is
under investigation.

\newpage
\appendix

\section{Boundedness}
\label{sec:bounded}

\begin{lemma}
  \label{lem:bounded}
  Let $P$ be a distribution on $\R$, with differentiable probability
  density function $f(\cdot)$.  Suppose $|f'(x)| \leq K$, all $x$.
  Then $f(x) \leq 1+2K$.
\end{lemma}

\begin{proof}
  Define
  \begin{eqnarray*}
    A_- &:=& \{x: f(x) \leq 1\}\\
    A_n &:=& \{x: 2^n < f(x) \leq 2^{n+1}\} \quad(n=0,1,\ldots).
  \end{eqnarray*}
  Then $\R$ is the disjoint union of these sets.

  We have $1 \geq P(A_n) \geq 2^n \lambda(A_n)$, where $\lambda$ is
  Lebesgue measure.  So $\lambda(A_n) \leq 2^{-n}$.

  On $A_n$, the total variation of $f$ does not exceed
  $K\times\lambda(A_n) \leq K \times 2^{-n}$.  Hence the total
  variation outside $A_-$ is at most $K \times \sum_0^\infty 2^{-n} =
  2K$.
\end{proof}

\end{document}